\documentclass[oneside, a4paper]{amsart}

\usepackage{config}
\usepackage{makros}

\title{Jordan-Chevalley decompositions over imperfect fields}

\date{}

\author{Fabian Hebestreit}
\address{Fabian Hebestreit: Fakultät für Mathematik, Universität Bielefeld, D-33501 Bielefeld}
\email{hebestreit@math.uni-bielefeld.de}

\author{Manuel Hoff}
\address{Manuel Hoff: Fakultät für Mathematik, Universität Bielefeld, D-33501 Bielefeld}
\email{manuel.hoff@uni-bielefeld.de}

\author{Werner Hoffmann}
\address{Werner Hoffmann: Fakultät für Mathematik, Universität Bielefeld, D-33501 Bielefeld}
\email{hoffmann@math.uni-bielefeld.de}

\begin{document}

\begin{abstract}
    We give a classification of Jordan-Chevalley decompositions of an endomorphism of a finite-dimensional vector space over a not necessarily perfect field, i.e.\ additive decompositions into commuting semisimple and nilpotent endomorphisms.
\end{abstract}

\maketitle

\tableofcontents

\section{Introduction} \label{sec:intro}

\subsection{Motivation} \label{subsec:motivation}

Let $K$ be a field and let $x \colon V \to V$ be an endomorphism of a finite-dimensional $K$-vector space $V$.
If the minimal polynomial $\minpol_x \in K[T]$ splits into linear factors then $x$ admits a Jordan normal form.
Decomposing the representing matrix into its diagonal and off-diagonal parts yields a decomposition $x = s + n$, where $s \colon V \to V$ is diagonalisable, $n \colon V \to V$ is nilpotent and $s \circ n = n \circ s$.
And while neither the basis for a Jordan form is unique, nor the Jordan form independent of the choice of basis, the decomposition into commuting diagonalisable and nilpotent endomorphisms \emph{does} turn out to be unique.

When $\minpol_x$ splits into linear factors only in $M[T]$ for some finite Galois extension $M/K$ then $x$ does not admit a Jordan normal form anymore.
Still, the base change $\id \otimes x \colon M \otimes_K V \to M \otimes_K V$ does so and the resulting decomposition of $\id \otimes x$ is invariant under the Galois action by uniqueness and thus descends to a decomposition $x = s + n$.
The endomorphism $n$ inherits nilpotence from its base change, but of course $s$ need no longer be diagonalisable.
It does, however, retain the feature that letting the variable $S$ act on $V$ through $s$ turns $V$ into a semisimple $K[S]$-module; this property is equivalent to the minimal polynomial $\minpol_s$ being square-free.

It is customary to refer to a decomposition $x = s + n$ with commuting $s$ and $n$, $s$ semisimple in the sense above and $n$ nilpotent as a \emph{Jordan-Chevalley decomposition}, since it was Chevalley who proved (essentially) the following result:

\begin{theoremno}[{\cite{chevalley}}]
    Let $K$ be a field and let $x \colon V \to V$ be an endomorphism of a finite-dimensional $K$-vector space $V$. 

    If every irreducible factor of $\minpol_x$ is separable, then the endomorphism $x$ admits a unique Jordan-Chevalley decomposition.
    Furthermore, for every field extension $M/K$ this decomposition remains a Jordan-Chevalley decomposition (over $M$) after extension of scalars to $\id \otimes x \colon M \otimes_K V \to M \otimes_K V$. 

    Conversely, if $x$ admits a Jordan-Chevalley decomposition that is stable under base change in this sense, then every irreducible factor of $\minpol_x$ is separable.
\end{theoremno}

An endomorphism $s \colon V \to V$ with the property that every base change $\id \otimes s \colon M \otimes_K V \to M \otimes_K V$ is semisimple over $M$ for every field extension $M/K$ is called \emph{absolutely} or \emph{geometrically} semisimple by some authors (though we warn the reader that there also exists the convention that this concept is itself referred to as semisimplicity; we will never do so).
The characterisation of semisimplicity in terms of the minimal polynomial makes it clear that it is the separability of (the irreducible factors of) the minimal polynomial which distinguishes plain from absolute semisimplicity.

At any rate, Chevalley's theorem settles the existence and uniqueness question for Jordan-Chevalley decompositions entirely for perfect fields, e.g.\ finite fields and all fields of characteristic $0$.

\medskip

In the present note we want to record what happens in the remaining cases, for which we have not found any reference in the literature (though it is hard to imagine that our results are not yet known).

That the behaviour is more complicated is easily illustrated by the following well-known examples:
If $f \in K[T]$ is a monic irreducible polynomial, then letting $\rmC_f$ denote the companion matrix of $f$, i.e.\ the matrix representing the multiplication by $T$ with respect to the standard basis of $K[T]/f(T)$, consider the matrices
\[
    A \, = \,
    \begin{pmatrix}
        \rmC_f & 0 \\ 
        0 & \rmC_f
    \end{pmatrix}
    \, , \;
    B \, = \, \rmC_{f^2} \;
    \in \Mat_{2 \deg f}(K) \, .
\]
The matrix $A$ has the obvious Jordan-Chevalley decomposition $A = A + 0$ but if $f$ is inseparable then for any $0 \neq E \in \Mat_{\deg f}(K)$ commuting with $\rmC_f$
\[
    A \, = \,
    \begin{pmatrix}
        \rmC_f & E \\
        0 & \rmC_f
    \end{pmatrix}
    +
    \begin{pmatrix}
        0 & - E \\
        0 & 0
    \end{pmatrix}
\]
is a Jordan-Chevalley decomposition as well, because 
\[
    f
    \begin{pmatrix}
        \rmC_f & E \\
        0 & \rmC_f
    \end{pmatrix}
    \, = \,
    \begin{pmatrix}
        f(\rmC_f) & f'(\rmC_f) \cdot E \\
        0 & f(\rmC_f)
    \end{pmatrix}
    \, = \,
    0
\]
as $f' = 0$ by inseparability, which implies that the minimal polynomial of the left-hand matrix is still $f$.
On the other hand, though this takes a bit more set-up to deduce, $B$ does not admit any Jordan-Chevalley decomposition (still assuming $f$ to be inseparable).

One can therefore ask how to determine whether a given endomorphism $x$ admits a Jordan-Chevalley decomposition at all, and if so, how many, both in total and up to conjugations that fix $x$.

\subsection{Results} \label{subsec:results}

To state our results first note that the problem splits along the primary decomposition of the action of $x$ on $V$, i.e.\ any Jordan-Chevalley decomposition of $x$ respects the direct sum decomposition of $V$ into its $f$-primary subspaces with respect to $x$, for varying monic irreducible $f \in K[T]$, since the constituent endomorphisms $s$ and $n$ of a Jordan-Chevalley decomposition commute with $x$ by assumption.
Furthermore, $s$ and $n$ evidently become a Jordan-Chevalley decomposition on each primary space, and conversely Jordan-Chevalley decompositions on all primary subspaces clearly assemble to one on $V$.
Without loosing generality we will therefore only treat the case of $f$-primary endomorphisms for some fixed monic irreducible polynomial $f \in K[T]$, i.e.\ of those endomorphisms whose minimal polynomial is a power of $f$.

\medskip

We shall phrase our existence result in terms of the Frobenius (or rational) normal form of an endomorphism $x \colon V \to V$, or equivalently its invariant factors.
To this end, recall that there is a unique representing matrix for $x$, which is a block matrix, in which the blocks are companion matrices for a (descending) divisor chain of monic polynomials.
This matrix is the Frobenius normal form of $x$ and the occuring polynomials are its invariant factors, the first one being the minimal polynomial.
Under our standing assumption that the minimal polynomial of $x$ is a power of $f$, it thus follows that all invariant factors are powers of $f$, so we can reconstruct them simply from the partition of $\frac{\dim_K V}{\deg f}$ given by their multiplicities.
Here, by a partition of a natural number we mean an unordered one, and it will be convenient to identifiy partitions of $m$ with decreasing sequences of natural numbers that sum to $m$ (and thus in particular vanish for all but finitely many entries); it will also turn out convenient to let the indices of such sequences start at $1$.
We thus denote by $\Part$ the set of (not necessarily strictly) decreasing maps $\varphi \colon \NN_{> 0} \to \NN$ that eventually reach $0$ and by $\Part_m$ the subset consisting of those that add up to $m$.
We will write a partition $\varphi \in \Part$ also as $[\varphi_1, \dotsc, \varphi_r]$ if $\varphi_i = 0$ for $i > r$.
Given $x \colon V \to V$ as above we denote by $\inv x \in \Part_m$ the partition encoding the multiplicities of $f$ in the invariant factors of $x$, where we write $m = \frac{\dim_K V}{\deg f}$.
We thus have a map
\[
    \inv \colon \set[\big]{x \colon V \to V}{\text{$x$ $f$-primary}} \to \Part_m
\]
that becomes a bijection after taking orbits for the $\Aut_K(V)$-action by conjugation on the left.
For example, the matrices $A, B \in \Mat_{2 \deg f}(K)$ considered above satisfy $\inv A = [1, 1]$ and $\inv B = [2]$ and this clearly constitutes an exhaustive list in the case $m = 2$.
Our first result now reads as follows:

\begin{theoremintro} \label{thm:existence}
    Let $K$ be a field, $V$ a finite-dimensional $K$-vector space and $x \colon V \to V$ an $f$-primary endomorphism of $V$ for some monic irreducible polynomial $f \in K[T]$ of inseparability degree $q \in \NN_{> 0}$.
    Then $x$ admits a Jordan-Chevalley decomposition if and only if we have
    \[
        \inv_{(i - 1)q + 1} x \, \leq \, 1 + \inv_{i q} x
    \]
    for all $i > 0$. 
\end{theoremintro}

It is perhaps suprising that the existence of a Jordan-Chevalley decomposition depends on the polynomial $f$ only through its inseparability degree, but this phenomenon will indeed persist through all of our results.
Note also that the numerical criterion is trivially satisfied if the inseparability degree is $1$, i.e.\ if $f$ is separable. 

In the examples, $A$ evidently satisfies the criterion, but $B$ does not (unless $q = 1$).
For illustration, let us also discuss the case $m = 3$.
We use this opportunity to introduce a bit of notation.
Given a partition $\psi$ we shall write
\[
    \rmC_{f, \psi} = \diag(\rmC_{f^{\psi_1}}, \rmC_{f^{\psi_2}}, \dotsc) \, .
\]
Then $\rmC_{f, \psi}$ represents multiplication by $T$ on $K[T]/f(T)^{\psi_1} \oplus K[T]/f(T)^{\psi_2} \oplus \dotsb$ with respect to the standard basis and hence it is the unique matrix in Frobenius normal form with $\inv \rmC_{f, \psi} = \psi$.
In this notation $A = \rmC_{f, [1, 1]}$ and $B = \rmC_{f, [2]}$ and in the case $m = 3$ there are three possible Frobenius normal forms that are given by
\[
    \rmC_{f, [1, 1, 1]} \, = \,
    \begin{pmatrix}
        \rmC_f & 0 & 0 \\
        0 & \rmC_f & 0 \\
        0 & 0 & \rmC_f
    \end{pmatrix} \, ,
    \quad
    \rmC_{f, [2, 1]} \, = \,
    \begin{pmatrix}
        \rmC_{f^2}
        & 0 
        \\
        0 
        & \rmC_f  
    \end{pmatrix}
    \quad \text{and} \quad
    \rmC_{f, [3]} \, = \,
    \rmC_{f^3} \, .
\]
For $q = 1$ all three of these admit Jordan-Chevalley decompositions while for $q = 2$ only the first two do and for $q \geq 3$ only the first does.

\medskip

With the existence question settled, let us turn to uniqueness.
We saw in the example of the matrix $A$ above that Jordan-Chevalley decompositions may come in different types.
We will now codify these types, again using invariant factors.
Namely, given a Jordan-Chevalley decomposition $x = s + n$ of the $f$-primary endomorphism $x$, we observe that $f = \minpol_s$:
Expanding terms, one readily checks that for any polynomial $g \in K[T]$ that annihilates $x$ the endomorphism $g(s) = g(x-n)$ is a multiple of $n$, whence the minimal polynomial of $s$ divides that of $x$.
As it is also square-free by semisimplicity the claim follows. 
Thus, denoting by $L = K[S]/f(S)$ the field obtained by adjoining a root of $f$ to $K$, we obtain an $L$-vector space structure on $V$ by letting $S$ act through $s$, and since $s$ and $n$ commute, $n$ is then not just $K$- but also $L$-linear.
We denote by $\typ(s,n) \in \Part$ the partition encoding the multiplicities of the invariant factors of $n$ over $L$.
In fact, we have $\typ(s,n) \in \Part_m$ since the minimal polynomial of the nilpotent endomorphism $n$ is a power of the linear polynomial $T \in K[T]$ and $m = \frac{\dim_K V}{\deg f} = \frac{\dim_L V}{1}$.
Just like the partition induced by the invariant factors classifies endomorphisms up to conjugation, the type classifies Jordan-Chevalley decompositions up to conjugation in the sense that the map
\[
    \typ \colon \, \set[\Bigg]{s, n \colon V \to V}{\begin{gathered} \text{$s \circ n = n \circ s$,} \\ \text{$n$ nilpotent, $\minpol_s = f$} \end{gathered}} \, \longrightarrow \, \Part_m
\]
becomes bijective after taking orbits for the $\Aut_K(V)$-action by conjugation on the left.
Given a Jordan-Chevalley decomposition $x = s + n$ of type $\varphi$, then after choosing a Jordan $L$-basis of $V$ for $n$ and using the standard $K$-basis of $L$ we obtain a $K$-basis of $V$ such that $x$ is represented by the block matrix $\rmJ_{f, \varphi}$ with blocks of shape 
\[
    \NiceMatrixOptions{xdots/shorten=0.5em}
    \begin{pNiceMatrix}
        \rmC_f & \unitmatrix & & \\
        & \Ddots & \Ddots & \\
        & & \Ddots & \unitmatrix \\
        & & & \rmC_f
    \end{pNiceMatrix} \, ,
\]
the sizes and multiplicities of which are encoded by the partition $\varphi \in \Part_m$.
In the examples of Jordan-Chevalley decompositions of $A = \rmC_{f, [1, 1]}$ above (where we assumed $f$ to be inseparable) we have $\typ(A, 0) = [1, 1]$ and
\[
    \typ \roundbr*{
    \begin{pmatrix}
        \rmC_f & E \\
        0 & \rmC_f
    \end{pmatrix},
    \begin{pmatrix}
        0 & - E \\
        0 & 0
    \end{pmatrix}
    } \, = \, [2] \, .
\]

Now consider the diagram
\[
\begin{tikzcd}[column sep = large, row sep = large]
    \set[\Bigg]{s, n \colon V \to V}{\begin{gathered} \text{$s \circ n = n \circ s$,} \\ \text{$n$ nilpotent, $\minpol_s = f$} \end{gathered}}
    \ar[d, "{\substack{(s,n) \\ \text{\rotatebox[origin=c]{-90}{$\mapsto$}} \\s+n}}"]
    \ar[r,"\typ"]
    & \Part_m
    \ar[dashed,d,"\zeta_{f,m}"] \\
    \set[\Big]{x \colon V \to V}{\text{$x$ $f$-primary}}
    \ar[r,"\inv"]
    & \Part_m
\end{tikzcd}
\]
in which the right vertical map $\zeta_{f,m} \colon \Part_m \to \Part_m$ is induced by passing to $\Aut_K(V)$-orbits; as the notation suggests this map only depends on the isomorphism class of $V$, that is on $m$.
For a fixed $x$, extraction of types now induces a map
\[
    \typ \colon \, \curlybr[\big]{\text{Jordan-Chevalley decompositions $(s, n)$ of $x$}} \, \longrightarrow \, \zeta_{f, m}^{-1} \roundbr[\big]{\inv x}
\]
that again becomes a bijection after taking orbits for the $\Aut_K(V, x)$-action by conjugation on the left.
Our next task is thus to describe the map $\zeta_{f,m}$.

To this end recall that partitions can be spliced, resulting in maps $+ \colon \Part_m \times \Part_{m'} \to \Part_{m+m'}$ that give $\Part$ the structure of a commutative monoid; for example we have $[3, 1, 1] + [4, 2] = [4, 3, 2, 1, 1]$.
This splicing operation is compatible with taking direct sums of endomorphisms in the sense that 
\[
    \inv(x \oplus x') \, = \, \inv x + \inv x'
    \quad \text{and} \quad
    \typ(s\oplus s', n\oplus n') \, = \, \typ(s,n) + \typ(s',n') \, ,
\]
which implies that the map
\[
    \zeta_f \colon \, \Part \, \longrightarrow \, \Part \, ,
\]
given by $\zeta_{f, m}$ on $\Part_m$ for each $m$, is a homomorphism of commutative monoids.
Since $\Part$ is, as a commutative monoid, freely generated by the singleton partitions $[a]$ for $a \geq 1$, the homomorphism $\zeta_f$ is entirely determined by its values on these.
We have the following description:

\begin{theoremintro} \label{thm:uniqueness-type}
    Let $K$ be a field and $f \in K[T]$ a monic irreducible polynomial $f$ of inseparability degree $q \in \NN_{> 0}$.
    Then the homomorphism $\zeta_f \colon \Part \to \Part$ is given by
    \[
        [a] \, \longmapsto \, l \cdot [k + 1] + (q - l) \cdot [k] = \squarebr[\big]{\underbrace{k + 1, \dotsc, k+ 1}_l, \underbrace{k, \dotsc, k}_{q - l}}
    \]
    where $k, l \in \NN$ are the unique natural numbers with $a = k q + l$ and $l < q$.
\end{theoremintro}

In particular, $\zeta_f$ again depends on $f$ only through its inseparability degree.
For $q = 1$ the map $\zeta_f$ is the identity, or in other words the invariant factors of $x = s + n$ over $K$ and of $n$ over $L$ have the same multiplicities; the bijectivity of $\zeta_f$ in this case is of course implied by Chevalley's result.
For $q, m > 1$ the map $\zeta_{f, m} \colon \Part_m \to \Part_m$ is never injective, since for example the preimages of $m \cdot [1]$ are precisely those partitions $\varphi \in \Part_m$ with $\varphi_i \leq q$ for all $i$.
That is, if $f$ is inseparable then the matrix $\diag(\rmC_f, \dotsc, \rmC_f) \in \Mat_{m \deg f}(K)$ admits more than one type of Jordan-Chevalley decomposition for every $m > 1$.
In fact, it follows from \cref{thm:uniqueness-type} that if $x$ admits a Jordan-Chevalley decomposition then it admits more than one type thereof if and only if $q > 1$ and there exist $r \geq 0$ and $0 < i_1 < i_2$ such that we have
\[
    \inv_{(i_1 - 1)q + 1} x \, = \, \inv_{(i_1 - 1)q + 2} x \, = \, r + 1
    \quad \text{and} \quad
    \inv_{i_2 q - 1} x \, = \, \inv_{i_2 q} x \, = \, r \, ,
\]
see \cref{lem:zeta-one-preimage}.

Returning to the case $m = 3$, we find that for $q = 1$ all of $\rmC_{f, [1, 1, 1]}$, $\rmC_{f, [2, 1]}$ and $\rmC_{f, [3]}$ admit Jordan-Chevalley decomposition of a unique type each (as of course implied by Chevalley's theorem), for $q = 2$ we have 
\[
    \zeta_f \roundbr[\big]{[1, 1, 1]} \, = \, \zeta_f \roundbr[\big]{[2, 1]} \, = \, [1, 1, 1]
    \quad \text{and} \quad
    \zeta_f \roundbr[\big]{[3]} = [2, 1]
\]
so that in this case $\rmC_{f, [2, 1]}$ admits Jordan-Chevalley decompositions of a unique type, whereas those of $\rmC_{f, [1, 1, 1]}$ come in two types, and for $q \geq 3$ the map $\zeta_{f,3} \colon \Part_3 \to \Part_3$ is constant with value $[1, 1, 1]$ so that then $\rmC_{f, [1, 1, 1]}$ admits three types of Jordan-Chevalley decompositions.

\medskip

To complete the analysis, we turn to the uniqueness of Jordan-Chevalley decompositions within a given type.
To this end, we observe that one can view \cref{thm:uniqueness-type} as a classification of similarity types of the block matrices $\rmJ_{f, \varphi}$ introduced above:
Namely, an endomorphism $x$ admits a Jordan-Chevalley decomposition of type $\varphi$ if and only if it can be represented by the matrix $\rmJ_{f, \varphi}$.
Thus $\inv \rmJ_{f, \varphi} = \zeta_f(\varphi)$ and consequently two matrices $\rmJ_{f, \varphi}$ and $\rmJ_{f, \varphi'}$ are conjugate if and only if $\zeta_f(\varphi) = \zeta_f(\varphi')$.

Now, note that when $f$ is inseparable and $\varphi \neq m \cdot [1]$, then $\rmJ_{f, \varphi}$ admits not only the obvious Jordan-Chevalley decomposition of type $\varphi$ but in fact infinitely many:
As in the initial example of the matrix $A$, picking for every $i \geq 1$ such that $\varphi_i > 1$ a non-zero $E_i \in \Mat_{\deg f}(K)$ that commutes with $\rmC_f$ gives rise to a Jordan-Chevalley decomposition
\[
    \NiceMatrixOptions{xdots/shorten=0.5em}
    \rmJ_{f, [\varphi_i]} \, = \,
    \begin{pNiceMatrix}
        \rmC_f & & & E_i \\
        & \Ddots & & \\
        & & \Ddots & \\
        & & & \rmC_f
    \end{pNiceMatrix}
    \, + \,
    \begin{pNiceMatrix}
        0 & \unitmatrix & & -E_i \\
        & \Ddots & \Ddots & \\
        & & \Ddots & \unitmatrix \\
        & & & 0
    \end{pNiceMatrix}
\]
and these assemble into an additional Jordan-Chevalley decomposition of $\rmJ_{f, \varphi}$ that is still of type $\varphi$ (so long as $E_i \neq \unitmatrix$ whenever $\varphi_i = 2$).
Summarising, given an $f$-primary endomorphism $x$ with Jordan-Chevalley decomposition $x = s + n$, there are always infinitely many thereof of the same type unless $q = 1$ or $n = 0$.

To further quantify the Jordan-Chevalley decompositions of $x$ within a given type $\varphi \in \zeta_f^{-1}(\inv x)$ we note that they naturally carry the structure of an algebraic variety over $K$, and we calculate its dimension.
To make this precise, consider the functor
\[
    J(x, \varphi) \colon \, \catcalgs(K) \, \longrightarrow \, \catsets
\]
sending a commutative $K$-algebra $R$ to the set of decompositions $\id \otimes x = s + n$ in $\End_R(R \otimes_K V)$ for which there exists a faithfully flat ring homomorphism $R \to R'$ such that the pair $(\id \otimes s, \id \otimes n)$ is conjugate to $(\id \otimes s', \id \otimes n')$ in $\End_{R'}(R' \otimes_K V)$ for some (or equivalently any) Jordan-Chevalley decomposition $x = s' + n'$ in $\End_K(V)$ of type $\varphi$; in particular $J(x, \varphi)(K)$ is precisely the set of Jordan-Chevalley decompositions of $x$ of type $\varphi$.
We then have the following result:

\begin{theoremintro} \label{thm:uniqueness}
    Let $K$ be a field, $V$ a finite-dimensional $K$-vector space, $x \colon V \to V$ an $f$-primary endomorphism of $V$ for some monic irreducible polynomial $f \in K[T]$ of inseparability degree $q \in \NN_{> 0}$ and $\varphi \in \zeta_f^{-1}(\inv x)$.

    Then $J(x, \varphi)$ is representable by a quasi-affine, smooth and geometrically connected $K$-scheme whose set of $K$-points is dense and we have
    \[
        \dim J(x,\varphi) = \deg f \cdot \roundbr[\big]{\curlybr{\inv x} - \curlybr{\varphi}} \, ,
    \]
    where for a partition $\omega$ we write $\curlybr{\omega} \coloneqq \sum_i (2i - 1) \cdot \omega_i$.
\end{theoremintro}

In particular, this formula again implies the infinitude of the set of Jordan-Chevalley decompositions of $x$ outside the degenerate cases.

\subsection{Acknowledgments} \label{subsec:acknowledgments}

We would like to thank Lukas Kühne for helping with some of the combinatorics.

This work was funded by the Deutsche Forschungsgemeinschaft (DFG, German Research Foundation)
– Project-ID 491392403 – TRR 358.
\section{Proofs of the results}

We start by deducing the description of the homomorphism $\zeta_f \colon \Part \to \Part$ stated above:

\begin{proof}[Proof of \cref{thm:uniqueness-type}]
    Recall that we have defined the finite field extension $L = K[S]/f(S)$ of $K$.
    Let $a > 0$ and set $(V, x, s, n) \coloneqq (L[U]/U^a, S + U, S, U)$.
    Then $(s, n)$ is a Jordan-Chevalley decomposition of $x$ of type $[a]$ and therefore our goal is to compute $\inv x$.
    
    In the following we write $\mult_b(\psi)$ for the multiplicity of a number $b \geq 1$ in a partition $\psi \in \Part$.
    We then have the standard identity
    \[
        \mult_b(\inv x) \, = \, (k_b - k_{b - 1}) - (k_{b + 1} - k_b) \, ,
    \]
    where we write $k_i \coloneqq \frac{\dim_K \ker f^i(x)}{\deg f}$.
    Now note that by definition of the inseparability degree we have a decomposition $f = (T - S)^q \cdot g$ in $L[T]$ with $g(S) \neq 0$ in $L$.
    Thus, using the $L$-vector space structure on $V$, the $L$-linearity of $x$ and the observation that $g(x) \in \End_L(V)$ is invertible (because $x$ is $(T - S)$-primary as an $L$-linear endomorphism of $V$ and $g$ is coprime to $T - S$), we obtain
    \[
        k_i \, = \, \dim_L \ker f^i(x) \, = \, \dim_L \ker (x - S)^{i q} \, = \, \min \curlybr{i q, a}
    \]
    and consequently
    \[
        k_{i + 1} - k_i \, = \,
        \begin{cases}
            q & \text{if $i < k$} \, , \\
            l & \text{if $i = k$} \, , \\
            0 & \text{else}
        \end{cases}
        \qquad \text{and} \qquad
        \mult_b(\inv x) \, = \,
        \begin{cases}
            q - l & \text{if $b = k$} \, , \\
            l & \text{if $b = k + 1$} \, , \\
            0 & \text{else} \, ,
        \end{cases}
    \]
    where $a = k q + l$ as in the statement of the theorem.
\end{proof}

Next we give two combinatorial lemmas, the first of which in combination with \cref{thm:uniqueness-type} implies \cref{thm:existence}:

\begin{lemma} \label{lem:zeta-image}
    Let $q \in \NN_{> 0}$ and write $\zeta_q \colon \Part \to \Part$ for the homomorphism of commutative monoids that is given on generators by $[a] \mapsto l \cdot [k + 1] + (q - l) \cdot [k]$, where $k, l \in \NN$ are the unique natural numbers with $a = k q + l$ and $l < q$.
    Then a partition $\psi$ is contained in the image of $\zeta_q$ if and only if we have
    \[ \label{eq:zeta-image}
        \psi_{(i - 1) q + 1} \, \leq \, 1 + \psi_{i q}
        \tag{$\ast$}
    \]
    for all $i > 0$.
    In this case a preimage $\varphi \in \zeta_q^{-1}(\psi)$ is given by
    \[
        \varphi_i \, \coloneqq \, \sum_{j = (i - 1) q + 1}^{i q} \psi_j
    \]
    and in fact $\varphi$ is the unique preimage of $\psi$ such that for every $r \geq 0$ there exists at most one $i$ with the property that $r q < \varphi_i < (r + 1) q$.
\end{lemma}

\begin{proof}
    First assume that $\psi$ satisfies the condition \eqref{eq:zeta-image} and define $\varphi$ as in the statement of the theorem; note that the monotony of $\psi$ implies the monotony of $\varphi$.
    Then we see that $\zeta_q([\varphi_i]) = [\psi_{i q + 1}, \dotsc, \psi_{(i + 1) q}]$ and thus $\zeta_q(\varphi) = \psi$ as desired.

    Now conversely assume that $\psi$ is in the image of $\zeta_q$ and choose a preimage $\varphi$ such that $\abs{\set{i \geq 1}{q \nmid \varphi_i}}$ is minimal.
    For every $r \geq 0$ there is then at most one $i$ such that $r q < \varphi_i < (r + 1) q$:
    If there existed $i_1 \neq i_2$ with $r q < \varphi_{i_1}, \varphi_{i_2} < (r + 1)q$, then we would have
    \[
        rq \, < \, \varphi_{i_1} + \varphi_{i_2} - r q \, \leq \, (r + 1) q
        \quad \text{or} \quad
        rq \, \leq \, \varphi_{i_1} + \varphi_{i_2} - (r + 1)q \, < \, (r + 1) q
    \]
    and consequently
    \[
        \zeta_q \roundbr[\big]{[\varphi_{i_1}] + [\varphi_{i_2}]} \, = \, \zeta_q \roundbr[\big]{[rq] + [\varphi_{i_1} + \varphi_{i_2} - r q]} \quad \text{or} \quad
        \zeta_q \roundbr[\big]{[\varphi_{i_1}] + [\varphi_{i_2}]} \, = \, \zeta_q \roundbr[\big]{[(r + 1) q] + [\varphi_{i_1} + \varphi_{i_2} - (r + 1) q]} \, ,
    \]
    yielding a preimage of $\psi$ that would contradict the minimality of $\abs{\set{i \geq 1}{q \nmid \varphi_i}}$.
    But then it follows that
    \[
        [\psi_1, \psi_2, \dotsc] \, = \, \squarebr[\big]{\zeta_q([\varphi_1])_1, \dotsc, \, \zeta_q([\varphi_1])_q, \, \zeta_q([\varphi_2])_1, \dotsc, \zeta_q([\varphi_2])_q, \dotsc}
    \]
    so that condition \eqref{eq:zeta-image} is satisfied.
\end{proof}

\begin{lemma} \label{lem:zeta-one-preimage}
    Let $q \in \NN_{> 1}$ and $\zeta_q \colon \Part \to \Part$ as in \cref{lem:zeta-image}.
    Then a partition $\psi \in \im \zeta_q$ has more than one preimage under $\zeta_q$ if and only if there exist $r \geq 0$ and $0 < i_1 < i_2$ such that we have
    \[ \label{eq:zeta-one-preimage}
        \psi_{(i_1 - 1)q + 1} \, = \, \psi_{(i_1 - 1)q + 2} = r + 1
        \quad \text{and} \quad
        \psi_{i_2 q - 1} \, = \, \psi_{i_2 q} = r \, .
        \tag{$\ast \ast$}
    \]
\end{lemma}

\begin{proof}
    First assume that there exist $r \geq 0$ and $0 < i_1 < i_2$ satisfying the conditions \eqref{eq:zeta-one-preimage} and assume without loss of generality that $i_1$ is chosen maximal and then that $i_2$ is chosen minimal with respect to the given $r$.
    Now let $\varphi \in \zeta_q^{-1}(\psi)$ be the standard preimage from \cref{lem:zeta-image}.
    Then the conditions \eqref{eq:zeta-one-preimage} translate precisely into the properties
    \[
        rq \, \leq \, \varphi_{i_1} \, , \,  \varphi_{i_2} \, \leq \, (r + 1)q
        \quad \text{and} \quad
        \varphi_{i_1} - \varphi_{i_2} \geq 2
    \]
    while our mild extra hypothesis guarantees that $\varphi_{i_2} < \varphi_i < \varphi_{i_1}$ for all $i_1 < i < i_2$ (in fact we always have $i_2 - i_1 = 1$ when $q > 2$ and $i_2 - i_1 \leq 2$ when $q = 2$).
    It now follows that the partition $\varphi'$ defined by
    \[
        \varphi'_i \coloneqq
        \begin{cases}
            \varphi_i & \text{if $i \neq i_1, i_2$} \, , \\
            \varphi_{i_1} - 1 & \text{if $i = i_1$} \, , \\
            \varphi_{i_2} + 1 & \text{if $i = i_2$}
        \end{cases}
    \]
    is another preimage of $\psi$ under $\zeta_q$.

    Now conversely assume that there exists a preimage $\varphi' \in \zeta_q^{-1}(\psi)$ different from the standard one $\varphi$ from \cref{lem:zeta-image}.
    Then there exists $r \geq 0$ such that $rq < \varphi_i < (r + 1)q$ for more than one $i$.
    Now let $i_1$ be minimal with $\varphi'_{i_1} \leq (r + 1)q$ and let $i_2$ be maximal with $\varphi'_{i_2} \geq rq$.
    Then by the algorithm in the proof of \cref{lem:zeta-image} we also have $rq \leq \varphi_{i_1}, \varphi_{i_2} \leq (r + 1)q$ and $\varphi_{i_1} - \varphi_{i_2} \geq 2$ as desired.
\end{proof}

Lastly we give the proof of our dimension formula for the $J(x, \varphi)$:

\begin{proof}[Proof of \cref{thm:uniqueness}]
    Let us write $D(x) \colon \catcalgs(K) \to \catsets$ for the functor sending $R$ to the set of decompositions $x = s + n$ in $\End_R(R \otimes_K V)$; this is an affine space over $K$, isomorphic to $\affspace(\End_K(V))$ via any of the projections $(s, n) \mapsto s, n$, and we have $J(x, \varphi) \subseteq D(x)$.

    Now recall that the centraliser $\End_K(V, x) \subseteq \End_K(V)$ is a $K$-subalgebra of dimension $\deg f \cdot \curlybr{\inv x}$, see for example \cite[Section VIII.2]{gantmacher}.
    Thus $\intAut_K(V, x)$ is a smooth connected affine $K$-group scheme of the same dimension (and its $K$-points are dense as soon as the field $K$ is infinite, see for example \cite{mo-rational-points-affine-space}).
    We have a natural action of $\intAut_K(V, x)$ on $D(x)$ and we observe that $J(x, \varphi)$ is the orbit of any of its $K$-points under this action, see \cite[Section 7.c]{milne}.
    From this it follows that $J(x, \varphi)$ is a smooth and geometrically connected locally closed subscheme of $D(x)$ and that its $K$-points are dense whenever $K$ is infinite.

    Now the stabiliser $\Stab_{\intAut_K(V, x)}(s_0, n_0) = \intAut_K(V, x, s_0, n_0)$ identifies with the Weil restriction of scalars $\Res_{L/K} \intAut_L(V, n_0)$ and thus is of dimension $\deg f \cdot \curlybr{\varphi}$ by the same argument as before.
    From this we obtain the desired identity $\dim J(x, \varphi) = \deg f \cdot (\curlybr{\inv x} - \curlybr{\varphi})$.
    When $q = 1$ we see that $J(x, \varphi)$ is of dimension $0$ and consequently $J(x, \varphi) \cong \Spec(K)$ and when $q > 1$ the field $K$ is necessarily infinite; in both cases it follows that the $K$-points of $J(x, \varphi)$ are dense.
\end{proof}

\begin{example}
    Suppose that $x$ is already semisimple, i.e.\ we have $\minpol_x = f$.
    Then we have $\inv x = m \cdot [1]$ and consequently $\zeta_q^{-1}(\inv x)$ is the set of all partitions $\varphi \in \Part_m$ with $\varphi_i \leq q$ for all $i$.
    Fixing such $\varphi \in \zeta_q^{-1}(\inv x)$, giving a Jordan-Chevalley decomposition $x = s + n$ of type $\varphi$ then is equivalent to giving the nilpotent $L$-linear endomorphism $n$ of $V$ of type $\varphi$, where we give $V$ the $L$-vector space structure induced by $x$ (which makes sense by semisimplicity).

    In fact, we have an isomorphism
    \[
        J(x, \varphi) \, \to \, \Res_{L/K} N(V, \varphi) \, , \qquad (s, n) \, \mapsto \, n \, ,
    \]
    where $N(V, \varphi) \subseteq \affspace(\End_L(V))$ denotes the $\intAut_L(V)$-orbit of nilpotent endomorphisms of type $\varphi$.

    As a concrete example, in the simplest non-trivial case where $m = 2$ and $\varphi = [2]$ and after choosing an $L$-basis of $V$ we have
    \[
        N(V, \varphi) \, = \, \set[\Big]{\begin{psmallmatrix} a & b \\ c & -a \end{psmallmatrix}}{a^2 + bc = 0} \setminus \curlybr{0} \, \subseteq \, \Mat_{2, L} \, .
    \]
\end{example}

\begin{example}
    Suppose that $\deg f = q = 2$ and consider $(V, x)$ where $V$ is of dimension $6 \cdot 2 = 12$.
    Here is a table that lists the partitions $\psi$ of the integer $6$, their preimages $\varphi$ under $\zeta_2$ and the corresponding dimensions of $J(x, \varphi)$.
    \begin{center}
    \begin{tabular}{c | l}
        $\psi$
        & $\varphi \colon \dim$
        \\
        \hline\hline
        $[1, 1, 1, 1, 1, 1]$
        & $[1, 1, 1, 1, 1, 1] \colon 0 \, , \, [2, 1, 1, 1, 1] \colon 20 \, , \, [2, 2, 1, 1] \colon 32 \, , \, [2, 2, 2] \colon 36$
        \\
        $[2, 1, 1, 1, 1]$
        & $[3, 1, 1, 1] \colon 16 \, , \, [3, 2, 1] \colon 24$
        \\
        $[2, 2, 1, 1]$
        & $[3, 3] \colon 16 \, , \, [4, 1, 1] \colon 16 \, , \, [4, 2] \colon 20$
        \\
        $[2, 2, 2]$
        &
        \\
        $[3, 1, 1, 1]$
        &
        \\
        $[3, 2, 1]$
        & $[5, 1] \colon 12$
        \\
        $[3, 3]$
        & $[6] \colon 12$
        \\
        $[4, 1, 1]$
        &
        \\
        $[4, 2]$
        &
        \\
        $[5, 1]$
        &
        \\
        $[6]$
        &
    \end{tabular}
    \end{center}
\end{example}

\printbibliography

\end{document}